\documentclass[a4paper,11pt]{article}
\usepackage{amsfonts,amssymb,amsmath,amscd,amsthm}
\usepackage{pictexwd,dcpic}
\usepackage[hmargin=2.5cm,vmargin=3.5cm]{geometry}
\newcommand{\mb}{\mathbb}
\newcommand{\mc}{\mathcal}

\newcommand{\lra}{\longrightarrow}
\newcommand{\torus}{\mathbb{T}}
\newcommand{\ntorus}{\mathbb{T}^{n}}
\newcommand{\Zn}{\mb{Z}^{n}}
\newcommand{\Rn}{\mb{R}^{n}}

\numberwithin{equation}{section}

\author{Amandip Sangha}
\title{KK-fibrations arising from Rieffel deformations}
\date{}

\begin{document}
\maketitle

\abstract{
The bundle map $\pi_{h}: \Gamma( (A_{tJ})_{t\in [0,1]} )\lra A_{hJ}$, for every $h\in [0,1]$, of the continuous field $(A_{tJ})_{t\in [0,1]}$ associated to the Rieffel deformation $A_{J}$ of a C*-algebra $A$ is shown to be a KK-equivalence by using a 2-cocycle twisting approach and RKK-fibrations.}
\tableofcontents

\theoremstyle{plain}
\newtheorem{theorem}{Theorem}[section]
\newtheorem{lemma}[theorem]{Lemma}
\newtheorem{corollary}[theorem]{Corollary}
\newtheorem{proposition}[theorem]{Proposition}

\theoremstyle{definition}
\newtheorem{definition}[theorem]{Definition}
\newtheorem{question}[theorem]{Question}
\newtheorem{remark}[theorem]{Remark}

\section{Introduction}
In \cite{Rieffel1} M. A. Rieffel introduced a C*-algebraic framework for deformation quantization whereby a C*-algebra $A$ equipped with an action of $\Rn$ by automorphisms and further supplied with a skew-symmetric $J\in M_{n}(\mb{R})$, produces a C*-algebra $A_{J}$ with multiplication $\times_{J}$, 
often referred to as the \textit{Rieffel deformation} of the original algebra. Several other well-known examples of C*-algebras can be shown to arise in this way.
The K-theory of the deformed algebra was studied in \cite{Rieffel3}, revealing that the deformed algebra $A_{J}$ and the original algebra $A$ have the same K-groups. There, the key technique was to show that $A_{J}$ was strongly Morita equivalent to a certain crossed product of (a stabilization and suspension of) $A$ by $\Rn$, followed by an application of the Connes-Thom result in K-theory, stability and Morita invariance of the K-functor.

Some operator algebraic approaches to deformation quantization use various notions of ``twists'', utilizing an action (e.g. of a group) combined with a distinguished element satisfying some cocyclicity-condition (e.g. a group 2-cocycle) as ingredients towards deforming a given algebra equipped with said action. One such procedure is explored by Kasprzak in \cite{kasprzak} where a locally compact abelian group $G$ acts on a C*-algebra $A$. Given a 2-cocycle $\psi$ on the dual group $\widehat{G}$, there is a method for obtaining a deformed algebra $A^{\psi}$. This procedure encompasses in particular Rieffel deformation as the case $G=\Rn$ with a certain choice of 2-cocycle on $\widehat{\Rn}$. Concerning K-theory, there is an isomorphism $A^{\psi}\rtimes G\cong A\rtimes G$ of crossed products which, for the case $G = \Rn$, when combined with the Connes-Thom result yields an identification of the K-groups of the deformed and undeformed algebras respectively.

The present paper discusses the continuous field over $[0,1]$ of the Rieffel deformation and shows that the evaluation map is a KK-equivalence. Namely, for a C*-algebra $A$ with an action of $\Rn$ and given a skew-symmetric matrix $J$, taking $t\in [0,1]$ and using $tJ$ as the skew-symmetric matrix gives the Rieffel deformation $A_{tJ}$. This will constitute a continuous field $(A_{tJ})_{t\in [0,1]}$ as was already explored in the original monograph \cite{Rieffel1}. We show that the evaluation map of the bundle algebra $\pi_{h}: \Gamma( (A_{tJ})_{t\in [0,1]} )\lra A_{hJ}$, for each $h\in [0,1]$, is a KK-equivalence. To accomplish this we shall employ the deformation approach of Kasprzak and consider a deformed bundle algebra $B^{\psi}$ which will be a $C([0,1])$-algebra equipped with a fibrewise action. As such, RKK-theory naturally enters and we show that $B^{\psi}$ is an \textit{RKK-fibration} in the sense of \cite{Echterhoff-Nest-Oyono-2} by appealing to the fibrewise action and the Connes-Thom result in RKK-theory. The important consequence of being an RKK-fibration here, is that the evaluation map of the $C([0,1])$-algebra becomes a KK-equivalence. Finally, the deformed bundle algebra $B^{\psi}$ will be shown to be $C([0,1])$-linearly *-isomorphic to the bundle algebra  $\Gamma( (A_{tJ})_{t\in [0,1]} )$ of the continuous field of the Rieffel deformation, thus yielding the promised result.

We now give a more specific outline of the paper. Section $\ref{KasprzakApproach}$ explains the approach to deformation taken in \cite{kasprzak}, where one starts with the action of a locally compact abelian group $G$ with a 2-cocycle $\psi$ on the Pontryagin dual $\widehat{G}$. A certain subalgebra $A^{\psi}\subseteq M(A\rtimes G)$ is obtained as the Landstad algebra of the G-product $(A\rtimes G, \lambda, \widehat{\alpha})$. After presenting the basic preliminaries and some of the needed results, we specialize to $G=\Rn$ with our specific 2-cocycle $\psi_{J}$. Section $\ref{sectionBundleAndRKK}$ discusses the relevant bundle and collects a few needed ingredients from \cite{Echterhoff-Nest-Oyono-2} on RKK-fibrations and their relation to KK-equivalences, and then proceeds to establish that the aforementioned bundle is an RKK-fibration. Section $\ref{RieffelField}$ recalls the main notions of Rieffel deformation, the associated continuous field and the relation to the 2-cocycle deformation. The main result regarding the evaluation map of the bundle algebra of the continuous field is then achieved as a consequence of the RKK-fibration laid forth in the preceding section.
In section 5 we comment on the special case called theta deformation, in which the action is not by $\Rn$ but $\ntorus$. There, a different bundle algebra is plausible. Namely, taking a fix-point algebra description of the deformed algebra, we use a strong Morita equivalence to a certain crossed product algebra by the integers and work with an integer crossed product bundle algebra. One is able to show that the related bundle evaluation map has a KK-contractible kernel by applying the Pimsner-Voiculescu six-term exact sequence, so the KK-equivalence follows. Finally we describe the invariance of the index pairing which we understand as a KK-product between elements of the K-group with (in particular) the Fredholm module coming from a spectral triple.

\section{Twisting by a 2-cocycle}\label{KasprzakApproach}
We recall the approach to deformation as in \cite{kasprzak}. The idea is based on twisting a dual C*-dynamical system by a 2-cocycle of the dual group. First we recollect some preliminaries on C*-dynamical systems and $G$-products (cf. \cite{pedersen} \S 7.8). 

\begin{definition} Let $G$ be a locally compact abelian group and $\widehat{G}$ its Pontryagin dual group. Let $B$ be a C*-algebra with a strict-continuous unitary-valued homomorphism $\lambda: G\lra M(B)$, and let $\widehat{\rho}$ be a strongly continuous action $\widehat{\rho}:\widehat{G}\lra Aut(B)$ satisfying
\[ \widehat{\rho}_{\chi}(\lambda_{\gamma}) = \chi(\gamma) \lambda_{\gamma} \]
for all $\chi\in \widehat{G}$ and $\gamma\in G$.
The triple $(B,\lambda,\widehat{\rho})$ is called a \textit{G-product}. One also simply refers to $B$ as a $G$-product when the rest is implicitly understood.
\end{definition}

Given a $G$-product $(B, \lambda,\widehat{\rho})$, one may extend the given unitary representation $\lambda$ to the *-homomorphism $\lambda: C^{*}(G)\lra M(B)$. Using the Fourier transform to identify $C^{*}(G)\cong C_{0}(\widehat{G})$ we write $\lambda: C_{0}(\widehat{G})\lra M(B)$. This map is injective and we often omit $\lambda$ from the notation.

\begin{definition} Let $(B, \lambda,\widehat{\rho})$ be a $G$-product and let $x\in M(B)$. The element $x$ satisfies the \textit{Landstad conditions} if:
\begin{align*}
\mbox{(i)}\,\,& \widehat{\rho}_{\chi}(x) = x \mbox{ for all } \chi\in\widehat{G},\\
\mbox{(ii)}\,\,& \mbox{the map } G\ni\gamma \mapsto \lambda_{\gamma}x\lambda_{\gamma}^{*}\in M(B) \mbox{ is norm continuous},\\
\mbox{(iii)}\,\,& fxg\in B \mbox{ for all } f,g\in C_{0}(\widehat{G}).
\end{align*} 
The set of elements satisfying the Landstad conditions turns out to be a subalgebra in $M(B)$. We shall refer to this subalgebra as the \textit{Landstad algebra} of the $G$-product.
\end{definition}

The foremost example of a $G$-product is produced by the crossed product construction. Indeed, given an abelian C*-dynamical system $(B,G,\alpha)$, the triple $(B\rtimes_{\alpha}G, \lambda, \widehat{\alpha})$ is a $G$-product whose Landstad algebra is precisely $B$. The following result states that any $G$-product arises in this way.

\begin{theorem}\label{theo1}\cite[Theorem 7.8.8]{pedersen} A C*-algebra $B$ is a $G$-product $(B,\lambda,\widehat{\rho})$ if and only if there exists a C*-dynamical system $(C,G,\beta)$ for which $B \cong C\rtimes_{\beta}G$. The C*-dynamical system is unique up to covariant isomorphism, the C*-algebra $C$ is just the associated Landstad algebra and $\beta = Ad\,\lambda$.
\end{theorem}

Recall that a \textit{2-cocycle} $\psi$ on the abelian group $\widehat{G}$ is a continuous function 
\[ \psi: \widehat{G}\times\widehat{G}\lra \mb{T} \]
satisfying
\renewcommand{\labelenumi}{(\roman{enumi})}
\begin{enumerate}
\item $\psi(e,\chi) = \psi(\chi,e) = 1$ for all $\chi\in\widehat{G}$,
\item $\psi(\chi_{1}, \chi_{2}+\chi_{3})\psi(\chi_{2},\chi_{3}) = \psi(\chi_{1}+\chi_{2}, \chi_{3})\psi(\chi_{1},\chi_{2})$ for all $\chi_{1},\chi_{2},\chi_{3}\in\widehat{G}$
\end{enumerate}

Given an element $\chi\in\widehat{G}$, define the function $\psi_{\chi} \in C_{b}(\widehat{G})$ by 
\[ \psi_{\chi}(\widehat{\sigma}) = \psi(\chi,\widehat{\sigma}) \mbox{ for } \widehat{\sigma}\in\widehat{G}. \]
Observing that $C_{b}(\widehat{G}) = M(C_{0}(\widehat{G}))$, use the obvious extension $\lambda: C_{b}(\widehat{G}) \lra M(B)$ and obtain unitaries
\begin{equation}\label{unitaries}
U_{\chi} = \lambda(\psi_{\chi}) \in M(B).
\end{equation}
The 2-cocycle condition for $\psi$ implies the following commutation rule for these unitaries
\[ U_{\chi_{1}+\chi_{2}} = \bar{\psi}(\chi_{1},\chi_{2})U_{\chi_{1}}\widehat{\rho}_{\chi_{1}}(U_{\chi_{2}}) .\]

\begin{lemma}\label{lemma1}\cite[Theorem 3.1]{kasprzak} Let $(B,\lambda,\widehat{\rho})$ be a $G$-product and $\psi$ a 2-cocycle on $\widehat{G}$. Use the unitaries of $(\ref{unitaries})$ to define the strongly continuous action $\widehat{\rho}^{\psi}: \widehat{G}\lra Aut(B)$,
\[ \widehat{\rho}^{\psi}_{\chi}(b) = U_{\chi}^{*}\widehat{\rho}_{\chi}(b)U_{\chi} \]
for $\chi\in\widehat{G}$ and $b\in B$. Then $(B, \lambda, \widehat{\rho}^{\psi})$ is a $G$-product.
\end{lemma}

\begin{definition}[Kasprzak deformation]  Let $A$ be a separable C*-algebra with strongly continuous action $\alpha: G\lra Aut(A)$ of the locally compact abelian group $G$, and $\psi$ a 2-cocycle on $\widehat{G}$. The $G$-product $(A\rtimes_{\alpha}G,\lambda,\widehat{\alpha})$ gives rise to the $G$-product $(A\rtimes_{\alpha}G,\lambda,\widehat{\alpha}^{\psi})$ by Lemma $\ref{lemma1}$. The deformed algebra $A^{\psi}$ is by definition the Landstad algebra of the $G$-product $(A\rtimes_{\alpha}G,\lambda,\widehat{\alpha}^{\psi})$.
\end{definition}

An interesting result is obtained by considering the original action on the deformed algebra. Denote by $\alpha^{\psi}: G\lra Aut(A^{\psi})$ the action $\alpha_{g}^{\psi}(x) = \lambda_{g}x\lambda_{g}^{*}$. If we consider the crossed product of the C*-dynamical system $(A^{\psi}, G, \alpha^{\psi})$ we get
\begin{lemma}\label{KasprzakLandstadTheorem} $A^{\psi}\rtimes_{\alpha^{\psi}}G \cong A\rtimes_{\alpha}G$.
\end{lemma}
\begin{proof} The proof is a literal application of Theorem $\ref{theo1}$. Indeed, let $B = A\rtimes_{\alpha}G$ and consider the $G$-product $(B, \lambda, \widehat{\alpha}^{\psi}) = (A\rtimes_{\alpha}G, \lambda, \widehat{\alpha}^{\psi})$. The Landstad algebra of this $G$-product is what we have called $A^{\psi}$ by definition, which is the algebra $C = A^{\psi}$ referred to in Theorem $\ref{theo1}$. Furthermore $\alpha^{\psi} = Ad\,\lambda$, which is the action $\beta$ in that theorem. In other words the C*-dynamical system is $(C,G,\beta) = (A^{\psi}, G, \alpha^{\psi})$ and the theorem yields the isomorphism $B \cong C\rtimes_{\beta}G$, in our case $A\rtimes_{\alpha}G \cong A^{\psi}\rtimes_{\alpha^{\psi}}G$ as claimed.

Following \cite{pedersen} and \cite{Landstad}, we may further describe the *-isomorphism $A^{\psi}\rtimes_{\alpha^{\psi}}G \lra A\rtimes_{\alpha}G$ as mapping $y\otimes g \mapsto y\lambda_{g}$, for $y\in A^{\psi}$ and $g\in C_{c}(G)$.
\end{proof}

Let our separable C*-algebra $A$ be equipped with a strongly continuous action $\sigma:\Rn\lra Aut(A)$, and let $J\in M_{n}(\mb{R})$ be a skew-symmetric matrix. On $\Rn$ we consider the symmetric bicharacter
\[ e: \Rn\times\Rn\lra\torus \]
\[ e(u,v) = e^{2\pi i\, u\cdot v} \]
which gives the group isomorphism $\Rn \cong \widehat{\Rn}$ by $u\mapsto e^{1}_{u}$ where $e^{1}_{u}(v) = e(u,v)$. 
We use the 2-cocycle $\psi_{J}: \widehat{\Rn}\times\widehat{\Rn}\lra \torus$,
\begin{equation}\label{2cocycleEq}
\psi_{J}(e^{1}_{u},e^{1}_{v}) = e^{1}_{u}(J v) = e(u, Jv) = e^{2\pi i\, u\cdot Jv}.
\end{equation}
By Lemma $\ref{lemma1}$ the $\Rn$-product $(A\rtimes_{\sigma}\Rn, \lambda, \widehat{\sigma})$ combined with the 2-cocycle $\psi_{J}$ gives the $\Rn$-product $(A\rtimes_{\sigma}\Rn, \lambda, \widehat{\sigma}^{\psi_{J}})$, and the deformed algebra $A^{\psi_{J}}$ is the corresponding Landstad algebra.

\section{Bundle structure and RKK-fibration}\label{sectionBundleAndRKK}
Let $X$ be a locally compact Hausdorff space. A C*-algebra $B$ is called a $C_{0}(X)$-algebra (cf. \cite[1.5]{Kasparov1}) if there is a non-degenerate *-homomorphism $\Phi_{B}: C_{0}(X) \lra ZM(B)$. One also writes $fb = \Phi_{B}(f)b$, for $f\in C_{0}(X)$ and $b\in B$. For each $x\in X$, letting $I_{x} = \{f\in C_{0}(X):\, f(x) = 0\}$ be the ideal of functions vanishing at $x$, then $I_{x}B\subseteq B$ is an ideal and the quotient $B_{x} = B/(I_{x}B)$ is called the \textit{fiber} over $x$. The quotient map $q_{x}:B\lra B_{x}$ is also referred to as \textit{evaluation at} $x$.

Recall that we are considering a strongly continuous action $\sigma: \Rn\lra Aut(A)$ on a separable C*-algebra $A$, and a real skew-symmetric matrix $J$.
Let $B = C([0,1])\otimes A = C([0,1],A)$ be equipped with the obvious $C([0,1])$-algebra structure $\Phi_{B}: C([0,1])\lra ZM(B)$, $\Phi_{B}(f)(g\otimes a) = fg\otimes a$. Define the action $\beta:\Rn\lra Aut(B)$ 
\begin{equation}\label{actionBeta}
\beta_{x}(y)(s) = \sigma_{\sqrt{s}x}(y(s)),
\end{equation}
for $x\in\Rn$, $y\in B$, $s\in [0,1]$. Let $\psi = \psi_{J}$ be the 2-cocycle from $(\ref{2cocycleEq})$. Then the 2-cocycle deformation $B^{\psi}$ is by definition the Landstad algebra of the $\Rn$-product $(B\rtimes_{\beta}\Rn, \lambda, \widehat{\beta}^{\psi})$. Recall the action $\beta^{\psi}:\Rn\lra Aut(B^{\psi})$ from the remark preceding Lemma $\ref{KasprzakLandstadTheorem}$.

\begin{lemma} The deformed algebra $B^{\psi}$ is a $C([0,1])$-algebra and the action $\beta^{\psi}:\Rn\lra Aut(B^{\psi})$ is fiberwise. There is a $C([0,1])$-linear *-isomorphism
\[ B^{\psi}\rtimes_{\beta^{\psi}}\Rn \lra B\rtimes_{\beta}\Rn. \]
\end{lemma}
\begin{proof} Clearly the action $\beta$ on $B$ is $C([0,1])$-linear, i.e. for every $x\in\Rn$, $\beta_{x}(\Phi_{B}(f)y) = \Phi_{B}(f)\beta_{x}(y)$ for every $f\in C([0,1])$ and $y\in B$. This entails that $\Phi_{B\rtimes_{\beta}\Rn}: C([0,1])\lra ZM(B\rtimes_{\beta}\Rn)$ given by $(\Phi_{B\rtimes_{\beta}\Rn}(f)y)(v) = \Phi_{B}(f)(y(v))$ gives a $C([0,1])$-algebra structure on $B\rtimes_{\beta}\Rn$. Concerning the dual action $\widehat{\beta}:\widehat{\Rn}\lra Aut(B\rtimes_{\beta}\Rn)$, for each $w\in\widehat{\Rn}$ the canonically extended automorphism $\widehat{\beta}_{w}: M(B\rtimes_{\beta}\Rn) \lra M(B\rtimes_{\beta}\Rn)$ satisfies $\widehat{\beta}_{w}(\Phi_{B\rtimes_{\beta}\Rn}(f)) = \Phi_{B\rtimes_{\beta}\Rn}(f)$ for every $f\in C([0,1])$. It then follows
that for any $y\in M(B\rtimes_{\beta}\Rn)$, 
\begin{align*}
\widehat{\beta}^{\psi}_{w}(\Phi_{B\rtimes_{\beta}\Rn}(f)) &= U_{w}^{*}\widehat{\beta}_{w}(\Phi_{B\rtimes_{\beta}\Rn}(f))U_{w} = U_{w}^{*}\Phi_{B\rtimes_{\beta}\Rn}(f)U_{w} \\
&= \Phi_{B\rtimes_{\beta}\Rn}(f)U^{*}_{w}U_{w} = \Phi_{B\rtimes_{\beta}\Rn}(f),
\end{align*}
i.e. $\Phi_{B\rtimes_{\beta}\Rn}( C([0,1]) ) \subseteq M(B\rtimes_{\beta}\Rn)^{\widehat{\beta}^{\psi}} = B^{\psi}$. Combined with the fact that $\Phi_{B\rtimes_{\beta}\Rn}( C([0,1]) ) \subseteq ZM(B\rtimes_{\beta}\Rn)$, this entails that we may define $\Phi_{B^{\psi}} = \Phi_{B\rtimes_{\beta}\Rn}$ to obtain a $C([0,1])$-algebra structure on $B^{\psi}$.

The action $\beta^{\psi}:\Rn\lra Aut(B^{\psi})$ is $\beta^{\psi}_{x}(y) = \lambda_{x}y\lambda_{x}^{*}$, for $y\in B^{\psi}$, and so
\begin{align*}
\beta_{x}^{\psi}(\Phi_{B^{\psi}}(f)y) &= \beta_{x}^{\psi}(\Phi_{B\rtimes_{\beta}\Rn}(f)y) = \lambda_{x}(\Phi_{B\rtimes_{\beta}\Rn}(f)y)\lambda_{x}^{*} \\
&= \Phi_{B\rtimes_{\beta}\Rn}(f)\lambda_{x}y\lambda_{x}^{*} =  \Phi_{B^{\psi}}(f)\beta_{x}^{\psi}(y),
\end{align*}
i.e. the action $\beta^{\psi}$ is fiberwise and hence naturally makes the crossed product $B^{\psi}\rtimes_{\beta^{\psi}}\Rn$ a $C([0,1])$-algebra where $\Phi_{B^{\psi} \rtimes_{\beta^{\psi}}\Rn}:C([0,1])\lra ZM(B^{\psi} \rtimes_{\beta^{\psi}}\Rn)$ is given by the composition of $\Phi_{B^{\psi}}$ with the inclusion $M(B^{\psi})\subseteq M(B^{\psi} \rtimes_{\beta^{\psi}}\Rn)$. By Lemma $\ref{KasprzakLandstadTheorem}$
\begin{equation}\label{thirdProp}
B^{\psi} \rtimes_{\beta^{\psi}}\Rn \cong B\rtimes_{\beta}\Rn,
\end{equation}
and we claim this *-isomorphism to be $C([0,1])$-linear. Indeed, denote this *-isomorphism $S: B^{\psi} \rtimes_{\beta^{\psi}}\Rn \lra B\rtimes_{\beta}\Rn$, which by Lemma $\ref{KasprzakLandstadTheorem}$ can be described as $S(y\otimes g) = y\lambda_{g}$ for $y\in B^{\psi}$ and $g\in C_{c}(\Rn)$, and it follows that
\begin{align*}
S(\Phi_{B^{\psi}\rtimes_{\beta^{\psi}}\Rn}(f)(y\otimes g)) &= S(\Phi_{B^{\psi}}(f)y\otimes g) \\
&= \Phi_{B^{\psi}}(f)y\lambda_{g} \\
&= \Phi_{B\rtimes_{\beta}\Rn}(f)y\lambda_{g} \\
&= \Phi_{B\rtimes_{\beta}\Rn}(f)S(y\otimes g)
\end{align*}
 for any $f\in C([0,1])$, i.e. $S\circ \Phi_{B^{\psi}\rtimes_{\beta^{\psi}}\Rn} = \Phi_{B\rtimes_{\beta}\Rn}\circ S$.
\end{proof}

Let $f: Y\lra X$ be a continuous map between locally compact spaces. The pullback construction gives a $C_{0}(X)$-algebra structure on $C_{0}(Y)$, since $f^{*}: C_{0}(X) \lra C_{b}(Y)$ and $C_{b}(Y) = ZM(C_{0}(Y))$, we let $\Phi_{C_{0}(Y)}: C_{0}(X) \lra ZM(C_{0}(Y))$, $\Phi_{C_{0}(Y)}(k) = f^{*}(k)$ be the pointwise multiplication operator by the pullback
\[ \Phi_{C_{0}(Y)}(k)h = f^{*}(k)h \]
for $k\in C_{0}(X)$, $h\in C_{0}(Y)$.

Given a $C_{0}(X)$-algebra $B$, a locally compact space $Y$ and $f:Y\lra X$ a continuous map, the \textit{pullback} $f^{*}(B)$ of $B$ along $f$ is the $C_{0}(Y)$-algebra
\begin{equation}\label{pullbackDef}
f^{*}(B) = C_{0}(Y)\otimes_{C_{0}(X)}B.
\end{equation}
The balanced tensor product in $(\ref{pullbackDef})$ is by definition the quotient of $C_{0}(Y)\otimes B$ by the ideal generated by 
\[ \{ \Phi_{C_{0}(Y)}(k)g\otimes b - g\otimes \Phi_{B}(k)b \, \big| \quad g\in C_{0}(Y), \, b\in B, \, k\in C_{0}(X) \}. \]
The $C_{0}(Y)$-algebra structure on $f^{*}(B)$ is pointwise multiplication on the left, $\Phi_{f^{*}(B)}: C_{0}(Y) \lra ZM(f^{*}(B))$, $\Phi_{f^{*}(B)}(h)(g\otimes b) = hg\otimes b$, for $h, g\in C_{0}(Y)$ and $b\in B$. 
Note that the fiber $f^{*}(B)_{y}$ over $y\in Y$ is $B_{f(y)}$. Indeed, as in the balanced tensor product one has $I_{y}C_{0}(Y)\otimes_{C_{0}(X)}B = C_{0}(Y)\otimes_{C_{0}(X)}I_{f(y)}B$, then
\begin{align*}
 f^{*}(B)_{y} &= C_{0}(Y)\otimes_{C_{0}(X)}B/ I_{y}C_{0}(Y)\otimes_{C_{0}(X)}B \\
&= C_{0}(Y)\otimes_{C_{0}(X)}B/ C_{0}(Y)\otimes_{C_{0}(X)}I_{f(y)}B \\
&= B/ I_{f(y)}B = B_{f(y)}.
\end{align*}

Recall that given two graded, separable C*-algebras $A$ and $B$, the group $KK(A,B)$ is the set of Kasparov $A$-$B$-modules (also called Kasparov cycles) modulo an appropriate equivalence relation (e.g. homotopy equivalence). Briefly, a Kasparov $A$-$B$-module is a triple $(E,\phi,F)$ where $E$ is a countably generated right Hilbert $B$-module, $\phi: A\lra \mc{L}_{B}(E)$ is a *-homomorphism and $F\in \mc{L}_{B}(E)$ is a degree 1 operator such that $[F,\phi(a)]$, $(F^{2}-1)\phi(a)$ and $(F-F^{*})\phi(a)$ are elements of $\mc{K}_{B}(E)$ for any $a\in A$. 

The \textit{KK-product} is a bilinear map
\[ KK(A,D)\times KK(D,B) \lra KK(A,B) \]
\[ (\textbf{x},\textbf{y}) \mapsto \textbf{x}\textbf{y} \]
where $A$, $B$ and $D$ are separable (and $D$ is $\sigma$-unital) C*-algebras. There is a multiplicatively neutral element $\textbf{1}_{D}= [(D, id, 0)] \in KK(D,D)$ such that for any $\textbf{x}\in KK(A,D)$ and $\textbf{y}\in KK(D,B)$ one has $\textbf{x}\textbf{1}_{D} = \textbf{x}$ and $\textbf{1}_{D}\textbf{y} = \textbf{y}$.

An element $\textbf{x}\in KK(A,B)$ is called a \textit{KK-equivalence} if it is invertible with respect to the KK-product, i.e. if there exists an element $\textbf{y}\in KK(B,A)$ such that $\textbf{xy} = \textbf{1}_{A} \in KK(A,A)$ and $\textbf{yx} = \textbf{1}_{B}\in KK(B,B)$. 

Given a graded *-homomorphism $\phi: A\lra B$, then $(B,\phi,0)$ is the naturally associated Kasparov $A$-$B$-module. We say $\phi$ is a KK-equivalence if the corresponding element $[(B,\phi,0)]\in KK(A,B)$ is a KK-equivalence.

Regarding $C_{0}(X)$-algebras there is a further refinement of the KK-groups called RKK-groups (\cite{Kasparov1}). Namely, for two $C_{0}(X)$-algebras $A$ and $B$, the group $RKK(X; A,B)$ consists of Kasparov $A$-$B$-modules $(E,\phi,F)$ as before, only with the additional requirement
\begin{equation}\label{DefRKKrelation}
(fa)\cdot e\cdot b = a\cdot e\cdot (fb)
\end{equation}
for any $f\in C_{0}(X)$, $a\in A$, $b\in B$ and $e\in E$.

The notions $RKK(X; \cdot, \cdot)$-product and $RKK(X; \cdot, \cdot)$-equivalence are similar to those of the KK-counterpart.

We let $\Delta^{p}\subseteq\mb{R}^{p+1}$ denote the standard $p$-simplex.

\begin{definition} A $C_{0}(X)$-algebra $B$ is called a \textit{KK-fibration} if for every positive integer $p$, every continuous map $f:\Delta^{p}\lra X$ and every element $v\in\Delta^{p}$ the evaluation $q_{v}: f^{*}(B)\lra B_{f(v)}$ is a KK-equivalence.
\end{definition}

\begin{definition} A $C_{0}(X)$-algebra $B$ is called an \textit{RKK-fibration} if for every positive integer $p$, every continuous map $f:\Delta^{p}\lra X$ and every element $v\in\Delta^{p}$, $f^{*}(B)$ is $RKK(\Delta^{p}; \cdot,\cdot)$-equivalent to $C(\Delta^{p},B_{f(v)})$.
\end{definition}

\begin{remark}\label{trivialRKKfibration} Given a C*-algebra $A$, the canonical $C_{0}(X)$-algebra $B = C_{0}(X)\otimes A$ is an RKK-fibration. Indeed, given $f: \Delta^{p}\lra X$ and $v\in\Delta^{p}$, the pullback 
\[ f^{*}(B) =  C(\Delta^{p})\otimes_{C_{0}(X)}C_{0}(X)\otimes A \]
is $C(\Delta^{p})$-linearly *-isomorphic to $C(\Delta^{p},B_{f(v)}) = C(\Delta^{p})\otimes B_{f(v)} = C(\Delta^{p})\otimes A$ by the map
\[ h\otimes g\otimes a \mapsto \Phi_{C(\Delta^{p})}(g)h\otimes a = f^{*}(g)h\otimes a, \]
where $h\in C(\Delta^{p})$, $g\in C_{0}(X)$ and $a\in A$.
This implies the required $RKK(\Delta^{p}; \cdot,\cdot)$-equivalence.
\end{remark}
Note also that the property of being an RKK-fibration is preserved under RKK-equivalence.

The following observation (\cite[Remark 1.4]{Echterhoff-Nest-Oyono-2}) will be useful.
\begin{lemma}\label{RKKfibKK} An $RKK$-fibration is a $KK$-fibration.
\end{lemma}
\begin{proof} Suppose $B$ is an $RKK$-fibration, let $f: \Delta^{p}\lra X$ and $v\in \Delta^{p}$. Concisely put, we get the following comutative diagram in the KK category in which all arrows but the right vertical arrow are already known to be isomorphisms
\[
\begin{CD}
C(\Delta^{p}, B_{f(v)}) @>{\textbf{r}}>> f^{*}(B) \\
@V{ev_{v}}VV               @VV{q_{v}}V \\
B_{f(v)} @>>{r(v)}> B_{f(v)}
\end{CD}
\]
so it follows that the right vertical arrow $q_{v}$ must be an isomorphism as well.

In details, by assumption there exists an invertible element
\[ \textbf{r}\in RKK(\Delta^{p}; C(\Delta^{p},B_{f(v)}), f^{*}(B)). \]
Here $C(\Delta^{p},B_{f(v)}) = C(\Delta^{p})\otimes B_{f(v)}$ is the canonical $C(\Delta^{p})$-algebra with constant fiber $B_{f(v)}$ over each point of $\Delta^{p}$, its bundle projection map being just the evaluation $ev_{w}:C(\Delta^{p},B_{f(v)})\lra B_{f(v)}$, $ev_{w}(f\otimes b) = f(w)b$, for any $w\in\Delta^{p}$, and it gives in particular the KK-equivalence $[ev_{v}] \in KK(C(\Delta^{p},B_{f(v)}),B_{f(v)})$. Recall also that $f^{*}(B)$ has fiber $B_{f(v)}$ over the point $v\in\Delta^{p}$, denote this bundle projection map $q_{v}$. From the invertible element $\textbf{r}\in RKK(\Delta^{p}; C(\Delta^{p},B_{f(v)}), f^{*}(B))$ we get an invertible element $r(v)\in KK(B_{f(v)}, B_{f(v)})$ which implements the KK-equivalence between the fibers. It follows from
\[ [q_{v}]\cdot\textbf{r} = r(v)[ev_{v}] \]
that $[q_{v}] = r(v)[ev_{v}]\textbf{r}^{-1}$ is a KK-equivalence.
\end{proof}

Recall the Connes-Thom isomorphism in K-theory $K_{i}(A\rtimes_{\alpha}\mb{R}) \cong K_{i-1}(A)$, $i=0, 1$, where $\alpha\in Aut(A)$ is a continuous action. The analogous result in KK-theory establishes the existence of an invertible element $\textbf{t}_{\alpha}\in KK^{1}(A,A\rtimes_{\alpha}\mb{R}) = KK(SA,A\rtimes_{\alpha}\mb{R})$, \textit{the Thom element}. In other words, $A$ and $A\rtimes_{\alpha}\mb{R}$ are KK-equivalent with dimension shift $1$.
The case of an $\Rn$-action is handled by repeated application of the above, yielding a KK-equivalence with total dimension shift $n \,(mod\,2)$. In dealing with $C_{0}(X)$-algebras we shall make use of the following RKK-version of the Connes-Thom isomorphism (see \cite[$\S 4$]{Kasparov1})

\begin{theorem}\cite[Theorem 3.5]{Echterhoff-Nest-Oyono-1} \label{ConnesThomKasparov} Let  $A$ be a $C_{0}(X)$-algebra and $\alpha: \Rn\lra Aut(A)$ a fibrewise action. There exists an invertible element
\[ \textbf{t}_{\alpha} \in RKK^{n}(X; A, A\rtimes_{\alpha}\Rn). \]
Hence $A$ and $A\rtimes_{\alpha}\Rn$ are $RKK$-equivalent with dimension shift $n\,(mod\,2)$.
\end{theorem}

\begin{theorem}\label{theoremRKKfibration} $B^{\psi}$ is an RKK-fibration.
\end{theorem}
\begin{proof} It follows from Theorem $\ref{ConnesThomKasparov}$ that $B^{\psi}$ is $RKK([0,1]; \cdot,\cdot)$-equivalent, with dimension shift $n\,(mod\,2)$, to $B^{\psi}\rtimes_{\beta^{\psi}}\Rn$. By the isomorphism $(\ref{thirdProp})$ the latter algebra is $RKK([0,1]; \cdot,\cdot)$-equivalent to $B\rtimes_{\beta}\Rn$, which by Theorem $\ref{ConnesThomKasparov}$ again is $RKK([0,1]; \cdot,\cdot)$-equivalent, with another dimension shift $n\,(mod\,2)$, to $B = C([0,1])\otimes A$. The total dimension shift thus far is $2n\,(mod\,2) = 0$, i.e. the net effect being no dimension shift, so $B^{\psi}$ is plainly $RKK([0,1]; \cdot,\cdot)$-equivalent to $B$. Finally, the algebra $B = C([0,1])\otimes A$ is clearly an $RKK$-fibration (Remark $\ref{trivialRKKfibration}$), thus proving the claim.
\end{proof}

It follows from Theorem $\ref{theoremRKKfibration}$ and Lemma $\ref{RKKfibKK}$ that $B^{\psi}$ is a KK-fibration. Taking the identity function of the $1$-simplex, $f: \Delta^{1} = [0,1]\lra [0,1]$, $f(s) = s$, we conclude that the evaluation map $q_{s}: B^{\psi}\lra (B^{\psi})_{s}$ is a KK-equivalence. Although maybe not completely transparent thus far, it will be made clear in section $\ref{RieffelField}$ that $B^{\psi} \cong \Gamma( (A_{tJ})_{t\in [0,1]} )$ is the bundle algebra of the continuous field over $[0,1]$ of the Rieffel deformation and $(B^{\psi})_{s} \cong A_{sJ}$ is the fiber over the point $s\in [0,1]$.

\section{The continuous field of the Rieffel deformation}\label{RieffelField}
We briefly recall some of the basic facts from \cite{Rieffel1} concerning Rieffel deformation. Let  $\sigma:\Rn\lra Aut(A)$ be a strongly continuous action on a separable C*-algebra $A$, and $J\in M_{n}(\mb{R})$ a skew-symmetric matrix.
Let $\tau$ be the translation action on the Frechet space $C_{b}(\Rn,A)$ and let $C_{u}(\Rn,A)$ be the largest subspace on which $\tau$ is strongly continuous. Denote by $\mc{B}^{A} = \mc{B}^{A}(\Rn)\subseteq C_{u}(\Rn,A)$ the subalgebra of smooth elements for the action $\tau$. For any $F\in \mc{B}^{A}(\Rn\times \Rn)$ the integral
\[ \iint F(u,v) e^{2\pi i u\cdot v}\, du\,dv \]
exists, as shown in \cite[Chapter 1]{Rieffel1} by considerations of oscillatory integrals. For $f, g\in \mc{B}^{A}(\Rn)$, the function $(u,v)\mapsto \tau_{Ju}(f)(x)\tau_{v}(g)(x)$ is an element of $\mc{B}^{A}(\Rn\times \Rn)$ for each $x\in\Rn$, hence the following integral is well defined
\begin{equation}\label{deformedProdFunctions}
(f\times_{J}g)(x) = \iint \tau_{Ju}(f)(x)\tau_{v}(g)(x)e^{2\pi i (u\cdot v)},
\end{equation}
and it turns out $\times_{J}$ defines an associative product on $\mc{B}^{A}(\Rn)$, and we denote by $\mc{B}^{A}_{J} = (\mc{B}^{A}(\Rn), \times_{J})$ this algebra structure.
Let $\mc{S}^{A}\subseteq \mc{B}^{A}$ be the subspace of $A$-valued Schwartz functions. This is naturally a right Hilbert $A$-module for the $A$-valued inner product $\langle f,g\rangle_{A} = \int f(x)^{*}g(x)$. Considering the product $\times_{J}$, it turns out $\mc{S}^{A}_{J}$ is an ideal in $\mc{B}^{A}_{J}$, this still being compatible with the Hilbert C*-module structure. In this way $\mc{S}^{A}_{J}$ carries a representation $L = L^{J}$ of $\mc{B}^{A}_{J}$ by adjointable operators 
\[ L: \mc{B}^{A}_{J} \lra \mc{L}(\mc{S}^{A}_{J}) \]
\[ L_{f}(\xi) = f\times_{J}\xi, \quad f\in \mc{B}^{A}_{J}, \, \xi\in \mc{S}^{A}_{J}. \]

Let $A^{\infty}\subseteq A$ denote the dense *-subalgebra of smooth elements for the action $\sigma$. For $a, b\in A^{\infty}$, the function $(u,v)\mapsto \sigma_{Ju}\sigma_{v}(b)$ is an element of $\mc{B}^{A}(\Rn\times\Rn)$ and we may define
\[ a\times_{J}b = \iint \sigma_{Ju}(a)\sigma_{v}(b) e^{2\pi i \, u\cdot v} \, du\, dv. \]

The homomorphism $A \lra C_{u}(\Rn,A)$, $a\mapsto \widetilde{a}$, $\widetilde{a}(x) = \sigma_{x}(a)$, is equivariant for the respective actions $\sigma$ and $\tau$, thus maps $A^{\infty}\lra \mc{B}^{A}$. Moreover, $\widetilde{a\times_{J}b} = \widetilde{a}\times_{J}\widetilde{b}$, i.e. this is a homomorphism for the products $\times_{J}$. Thus $A^{\infty}$ is represented on $\mc{S}^{A}_{J}$, and we define a new norm $||\cdot||_{J}$ on $A^{\infty}$, $||a||_{J} = ||L_{\widetilde{a}}||$.

\begin{definition}[Rieffel deformation] Equip $A^{\infty}$ with the product $\times_{J}$ and the norm $||\cdot||_{J}$. This completion is denoted $A_{J}$ and is called the deformation of $A$ along $\sigma$ by $J$, or in short the \textit{Rieffel deformation} of $A$.
\end{definition}

Below we list some of the properties of the Rieffel deformation.
\renewcommand{\labelenumi}{(\roman{enumi})}
\begin{lemma}[Properties of the Rieffel deformation]\label{RieffelProperties} Let $A$ be a separable C*-algebra, $\sigma: \Rn \lra Aut(A)$ a strongly continuous action and $J\in M_{n}(\mb{R})$ such that $J^{t} = -J$.
\begin{enumerate}
  \item\label{RieffelProperties1} $\times_{J}$ is associative and the involution $*$ for $A$ is also an involution for $A_{J}$, which thus becomes a C*-algebra
  \item\label{RieffelProperties3} $a\times_{J}b = ab$ for $J=0$
  \item\label{RieffelProperties4} For every fixed point $a\in A^{\sigma}$, $a\times_{J}b = ab$ and $b\times_{J}a = ba$ for every $b\in A$
  \item\label{RieffelProperties5} $(A_{J})_{K} = A_{J+K}$ for any skew-symmetric $K\in Mat_{n}(\mb{R})$
  \item\label{RieffelProperties6} The action $\sigma$ is also an action on $A_{J}$, $\sigma: \Rn\lra Aut(A_{J})$. Moreover $(A_{J})^{\infty} = (A^{\infty})_{J}$
  \item\label{RieffelProperties7} The dense subalgebra $(A^{\infty})_{J}\subseteq A_{J}$ is stable under holomorphic functional calculus
  \item\label{RieffelProperties8} Given a $\sigma$-invariant ideal $I\subseteq A$, the equivariant short exact sequence 
\[
\begin{CD}
0 @>>> I @>>> A @>>> A/I @>>> 0
\end{CD}
\]
implies a short exact sequence
\[
\begin{CD}
0 @>>> I_{J} @>>> A_{J} @>>> (A/I)_{J} @>>> 0
\end{CD}
\]
\item\label{RieffelProperties9} For any $T\in M_{n}(\mb{R})$, define a new action $\sigma^{T}$ by $\sigma^{T}_{x}(a) = \sigma_{Tx}(a)$, for $x\in\Rn$, $a\in A$. Performing the deformation procedure for the action $\sigma^{T}$ and skew-symmetric matrix $J$, denote by $\times_{J}^{T}$ the deformed product so obtained. Then
\[ \times^{T}_{J} = \times_{TJT^{t}}. \]
\end{enumerate}
\end{lemma}

The equivalence between the Rieffel deformation $A_{J}$ and the 2-cocycle deformation $A^{\psi_{J}}$ is given by the *-isomorphism of the following lemma. Recall that one considers the $\Rn$-product $(A\rtimes_{\sigma}\Rn, \lambda, \widehat{\sigma}^{\psi_{J}})$, the 2-cocycle $\psi_{J}$ in $(\ref{2cocycleEq})$ and $A^{\psi_{J}}\subseteq M(A\rtimes_{\sigma}\Rn)$ is the subalgebra satisfying the Landstad conditions.

\begin{lemma}\label{RieffelKasprzakApproach} There is a *-isomorphism
\[ T: A^{\psi_{J}} \lra A_{J}. \]
\end{lemma}
\begin{proof} We refer the reader to \cite{H-M} for details, and give only the form of the *-isomorphism here.  Let $y\in C_{c}(\Rn, A^{\infty}) \subseteq A\rtimes_{\sigma}\Rn\subseteq M(A\rtimes_{\sigma}\Rn)$, and suppose $y\in A^{\psi_{J}}$, which means that $\widehat{\sigma}^{\psi_{J}}_{x}(y) = y$ for all $x\in\Rn$. The isomorphism $T$ is described on such elements by
\[ T(y) = \int_{\Rn} y(v) \, dv. \]
\end{proof}

We consider $B = C([0,1])\otimes A$ with the action $\beta$ as in $(\ref{actionBeta})$. Note that $\beta_{x}(y)(s) = \sigma^{\sqrt{s}I}_{x}(y(s))$ (see Lemma $\ref{RieffelProperties} (ix)$).
For every $x\in\Rn$, let $\overline{\beta}_{x}\in Aut(M(B))$ denote the canonical extension of $\beta_{x}$ to the multiplier algebra, namely for $L\in M(B)$, $\overline{\beta}_{x}(L)(b) = \beta_{x}(L(\beta_{-x}(b)))$, for $b\in B$. A quick calculation reveals that for every $f\in C([0,1])$, $\overline{\beta}_{x}(\Phi_{B}(f)) = \Phi_{B}(f)$, i.e. $\Phi_{B}(C([0,1])) \subseteq M(B)^{\overline{\beta}}$. It is also clear that  $\Phi_{B}(C([0,1])) \subseteq M(B)^{\infty}$. From the inclusion $B\subseteq M(B)$ as a $\beta$-invariant ideal we get $B_{J}\subseteq M(B)_{J}$ by Lemma $\ref{RieffelProperties} \,\, (vii)$, and working inside $M(B)_{J}$ get from Lemma $\ref{RieffelProperties}\,\,(iii)$
\begin{equation}\label{deformFieldStructEq}
\Phi_{B}(f)\times_{J}y = \Phi_{B}(f)y = y\Phi_{B}(f) = y\times_{J}\Phi_{B}(f)
\end{equation}
for $y\in B^{\infty}$ and $f\in C([0,1])$, as $\Phi_{B}(f)\in M(B)^{\overline{\beta}}$. This yields a $C([0,1])$-algebra structure on $B_{J}$, denoted $\Phi_{B_{J}}: C([0,1])\lra ZM(B_{J})$ given by $\Phi_{B_{J}}(f)y = \Phi_{B}(f)\times_{J}y = \Phi_{B}(f)y$. As such, $B_{J}$ is an \textit{essential} $C([0,1])$-module, i.e. $\overline{C([0,1])B_{J}} = B_{J}$.

\begin{theorem} $(A_{tJ})_{t\in [0,1]}$ is a continuous field of C*-algebras, where we take as the algebra of sections $\Gamma( (A_{tJ})_{t\in [0,1]} )$ to be the algebra $B_{J}$.
\end{theorem}
\begin{proof} For each $s\in [0,1]$ let $K^{s} = I_{s}\otimes A$ be the ideal consisting of elements of $B = C([0,1],A)$ which vanish at the point $s$. Clearly, $B/K^{s} = A$. The short exact sequence
\[
\begin{CD}
0 @>>> K^{s} @>>> B @>>> A @>>> 0
\end{CD}
\]
is equivariant for $\beta$ acting on $K^{s}$ and $B$, and $\sigma^{\sqrt{s} 1}$ acting on $A$, so by Lemma $\ref{RieffelProperties}\,\,(vii)$ (cf. also Theorem 7.7 of \cite{Rieffel1}) we get a short exact sequence
\[
\begin{CD}
0 @>>> {K^{s}}_{J} @>>> B_{J} @>>> A^{\sqrt{s}1} _{J} @>>> 0
\end{CD}
\]

The fiber $(B_{J})_{s}$ over $s\in [0,1]$ of the $C([0,1])$-algebra $B_{J}$ is by definition the quotient $(B_{J})_{s} = B_{J}/(I_{s}B_{J})$. It is shown in \cite{Rieffel1} that ${K^{s}}_{J} = I_{s}B_{J}$, consequently $(B_{J})_{s} = B_{J}/(I_{s}B_{J}) = B_{J}/{K^{s}}_{J} = A^{\sqrt{s}1}_{J}$. Moreover, from Lemma $\ref{RieffelProperties}\,\,(viii)$ it follows that $A^{\sqrt{s}1}_{J} = A_{\sqrt{s}1J\sqrt{s}1} = A_{sJ}$, thus the bundle projection is $\pi_{s}: B_{J}\lra A_{sJ}$. Theorem 8.3 of \cite{Rieffel1} (see also Proposition 1.2 of \cite{Rieffel2}) establishes the continuity of the field $(A_{tJ})_{t\in [0,1]}$, for which $B_{J}$ is a maximal algebra of cross sections, henceforth denoted $\Gamma( (A_{tJ})_{t\in [0,1]} )$.
\end{proof}

Considering the *-isomorphism of Lemma $\ref{RieffelKasprzakApproach}$ at the level of bundles, we get
\begin{lemma}\label{Bundle-RieffelIsoKasprzak} The *-isomorphism 
\[ T: B^{\psi} \lra B_{J} \]
is $C([0,1])$-equivariant, i.e. $T\circ\Phi_{B^{\psi}} = \Phi_{B_{J}}\circ T$.
\end{lemma}
\begin{proof} Let $b\in C_{c}(\Rn, B^{\infty}) \subset B\rtimes_{\beta}\Rn \subset M(B\rtimes_{\beta}\Rn)$ be an element such that $\widehat{\beta}^{\psi}_{x}(b) = b$ for all $x\in\Rn$, i.e. $b$ is an element of $B^{\psi}$. The *-isomorphism is described on such elements by
\[ T(b) = \int_{\Rn}b(v). \]
Furthermore
\begin{align*}
T(\Phi_{B^{\psi}}(f)b) &= \int_{\Rn} (\Phi_{B^{\psi}}(f)b)(v) \,dv = \int_{\Rn} (\Phi_{B\rtimes_{\beta}\Rn}(f)b)(v)\, dv \\
&= \int_{\Rn} \Phi_{B}(f)(b(v))\, dv = \Phi_{B}(f) \int_{\Rn} b(v) \, dv,
\end{align*}
and since $\Phi_{B_{J}} = \Phi_{B}$ as in $(\ref{deformFieldStructEq})$, the claim follows.
\end{proof}

\begin{theorem}\label{maintheorem} Let $h\in [0,1]$. The evaluation map
\[ \pi_{h}: \Gamma( (A_{tJ})_{t\in [0,1]} ) \lra A_{hJ} \]
is a KK-equivalence.
\end{theorem}
\begin{proof} As $\Gamma( (A_{tJ})_{t\in [0,1]} ) = B_{J}$ is $C([0,1])$-linearly *-isomorphic to $B^{\psi}$, and $B^{\psi}$ is an RKK-fibration (Theorem $\ref{theoremRKKfibration}$), thus $\Gamma( (A_{tJ})_{t\in [0,1]} )$ is an $RKK$-fibration and hence a $KK$-fibration (Lemma $\ref{RKKfibKK}$). So for any $f: \Delta^{p}\lra [0,1]$ and $v\in\Delta^{p}$, the quotient map $q_{v}: f^{*}(\Gamma( (A_{tJ})_{t} )) \lra A_{f(v)J}$ is a $KK$-equivalence. We take the identity function of the $1$-simplex, namely $f: \Delta^{1}=[0,1]\lra [0,1]$, $f(s) = s$. Then $f^{*}(\Gamma( (A_{tJ})_{t} )) = \Gamma( (A_{tJ})_{t} )$, $q_{h} = \pi_{h}$ and
\[ \pi_{h}: \Gamma( (A_{tJ})_{t} ) \lra A_{hJ } \]
is a $KK$-equivalence, for every $h\in [0,1]$. 
\end{proof}

\section{Comments}
\subsection{Theta deformation}
Here we discuss a special case of Rieffel deformation, namely \textit{theta deformation} and one possible variation to the above approach to KK-equivalence by bundle methods. Theta deformation concerns a separable C*-algebra $A$ on which there is a strongly continuous action of the $n$-torus, $\sigma: \ntorus\lra Aut(A)$, with a given skew-symmetric matrix $\theta \in M_{n}(\mb{R})$. This is just a special case of Rieffel deformation in which the $n$-torus is regarded as the quotient $\ntorus = \Rn/2\pi\Zn$, and one obtains the deformed algebra $A_{\theta}$. An alternative and perhaps more direct picture can be given by following \cite{Connes-DuboisViolette}. First define $C(\ntorus_{\theta})$ to be the unital C*-algebra generated by  unitaries $u_{1},\ldots, u_{n}$ with relations
\[ u_{j}u_{k} = e^{2\pi i \theta_{j,k}}u_{k}u_{j}, \mbox{ for } j, k = 1,\ldots, n. \]
(Note that this is just the Rieffel deformation $C(\ntorus)_{\theta}$ of the commutative C*-algebra $C(\ntorus)$ with respect to the translation action of the $n$-torus; the notation $C(\ntorus_{\theta})$ is suggestive of the terminology of "noncommutative manifolds" as in \cite{Connes-DuboisViolette}). On $C(\ntorus_{\theta})$ there is the action $\tau: \ntorus \lra Aut(C(\ntorus_{\theta}))$, $\tau_{s}(u_{j}) = e^{2\pi i s_{j}}u_{j}$, for $s\in\ntorus$. By considering the diagonal action $\sigma\otimes\tau^{-1}: \ntorus\lra Aut(A\otimes C(\ntorus_{\theta}))$ one defines the theta deformed algebra
\begin{equation}\label{thetaDeformation}
A_{\theta} = (A\otimes C(\ntorus_{\theta}))^{\sigma\otimes\tau^{-1}}
\end{equation}
as the fixed-point C*-subalgebra for this diagonal action.

We shall define a continuous C*-bundle over $[0,1]$ whose fiber over $t\in [0,1]$ will not be $A_{t\theta}$ per se, but will be strongly Morita equivalent to it. The benefit of this particular bundle will be that the evaluation map will easily be seen to yield a KK-equivalence element, and the remaining KK-equivalence is then given by the strong Morita equivalence. First we record the result we need regarding the strong Morita equivalence.

\begin{lemma} $A_{\theta} \sim_{M} A\rtimes_{\sigma}\ntorus\rtimes_{\gamma_{1}}\mb{Z}\rtimes\cdots\rtimes_{\gamma_{n}}\mb{Z}$.
\end{lemma}
\begin{proof} By results of $\cite{Ng}$ we get the strong Morita equivalence
\[ (A\otimes C(\ntorus_{\theta}))^{\sigma\otimes\tau^{-1}} \sim_{M} (A\otimes C(\ntorus_{\theta}))\rtimes_{\sigma\otimes\tau^{-1}}\ntorus. \]
The latter crossed product algebra is *-isomorphic to the crossed product in the statement of the lemma, which we now define. Let $\gamma_{1} \in Aut(A\rtimes_{\sigma}\ntorus)$ be $\gamma_{1}(g)(s) = e^{2\pi i s_{1}}g(s)$ for $g\in A\rtimes_{\sigma}\ntorus$ and let $u_{1}$ be the implementing unitary. Proceed inductively to define actions $\gamma_{2},\ldots,\gamma_{n}$ with implementing unitaries $u_{2},\ldots,u_{n}$ such that 
\begin{equation}\label{GammaActions}
\gamma_{j}(g)(s) = e^{2\pi i s_{j}}g(s), \quad \gamma_{j}(u_{k}) = e^{2\pi i \theta_{j,k}}u_{k}, \, j<k,
\end{equation}
so the covariance relation $\gamma_{j}(u_{k}) = u_{j}u_{k}u_{j}^{*} = e^{2\pi i \theta_{j,k}}u_{k}$ means precisely $u_{j}u_{k} = e^{2\pi i \theta_{j,k}}u_{k}u_{j}$. The *-isomorphism $ (A\otimes C(\ntorus_{\theta}))\rtimes_{\sigma\otimes\tau^{-1}}\ntorus \lra A\rtimes_{\sigma}\ntorus\rtimes_{\gamma_{1}}\mb{Z}\rtimes\cdots\rtimes_{\gamma_{n}}\mb{Z}$ can be explicitly described on the dense *-subalgebra $A\otimes C(\ntorus_{\theta})\otimes C(\ntorus)$ as $a\otimes u_{j}\otimes h \mapsto u_{j}(ah)$ where one understands $ah\in A\otimes C(\ntorus)\subseteq A\rtimes_{\sigma}\ntorus$.
\end{proof}

Let $B = C([0,1])\otimes A\rtimes_{\sigma}\ntorus = C([0,1], A\rtimes_{\sigma}\ntorus)$. We may decompose $\sigma$ into its coordinate actions $\sigma_{1},\ldots,\sigma_{n}$ where $\sigma_{j}(z) = \sigma_{(1,\ldots,z,\ldots,1)}$ for $z\in\torus$. For $j,k=1,\ldots,n$ let $h_{j,k}\in C([0,1])$ be the function
\[ h_{j,k}(t) = e^{2\pi i t\theta_{j,k}}. \]
Define $\alpha_{1}\in Aut(B)$ by
\[ \alpha_{1}(f\otimes g) = f\otimes \widehat{\sigma_{1}}(g), \quad f\in C([0,1]), g\in A\rtimes_{\sigma}\ntorus \]
and let $v_{1}$ be the unitary implementing $\alpha_{1}$ in $B\rtimes_{\alpha_{1}}\mb{Z}$. Define $\alpha_{2}\in Aut(B\rtimes_{\alpha_{1}}\mb{Z})$ by
\[ \alpha_{2}( (f\otimes g)v_{1}^{m} ) = (h_{1,2}^{m}f\otimes \widehat{\sigma_{2}}^{m}(g))v_{1}^{m}, \quad m\in\mb{Z}. \] 
Proceeding inductively we thus obtain actions $\alpha_{1},\ldots,\alpha_{n}$ with respective implementing unitaries $v_{1},\ldots,v_{n}$,
\[ \alpha_{k}( (f\otimes g)v_{j}^{m} ) = v_{k}( (f\otimes g)v_{j}^{m} )v_{k}^{*} = (h_{j,k}^{m}f\otimes\widehat{\sigma_{k}}^{m}(g))v_{j}^{m}. \]
Let $\pi_{t}: C([0,1])\otimes A\rtimes_{\sigma}\ntorus \lra A\rtimes_{\sigma}\ntorus$ be the evaluation map, $\pi_{t}(f\otimes g) = f(t)g$. 
For each $t\in [0,1]$, starting with $A\rtimes_{\sigma}\ntorus$ inductively define actions $\gamma_{1}^{t},\ldots,\gamma_{n}^{t}$ as in $(\ref{GammaActions})$ with respective unitaries $u_{1},\ldots,u_{n}$ such that 
\[ \gamma_{j}^{t}(g)(s) = e^{2\pi i s_{j}}g(s), \quad \gamma_{j}^{t}(u_{k}) = e^{2\pi i t\theta_{j,k}}u_{k}. \]
Note that the actions $\gamma_{j}$ of $(\ref{GammaActions})$ are just $\gamma_{j} = \gamma_{j}^{1}$ with $t=1$. Furthermore, $\pi_{t}\circ\alpha_{1} = \gamma_{1}^{t}\circ\pi_{t}$, i.e. $\pi_{t}$ is a $\mb{Z}$-equivariant *-homomorphism between the C*-dynamical systems and so passes to a *-homomorphism between the crossed products 
\begin{equation}\label{sec5eqBundleStep1}
\pi_{t}: (C([0,1])\otimes A\rtimes_{\sigma}\ntorus)\rtimes_{\alpha_{1}}\mb{Z} \lra A\rtimes_{\sigma}\ntorus\rtimes_{\gamma_{1}^{t}}\mb{Z},
\end{equation}
which is a continuous C*-bundle. Iterating this, one has $\pi_{t}\circ\alpha_{j} = \gamma_{j}^{t}\circ\pi_{t}$ for each $j=1,\ldots,n$, where $\pi_{t}$ is understood on the appropriate crossed product. Thus we get a continuous C*-bundle
\begin{equation}\label{sec5eqBundle}
\pi_{t}: C([0,1], A\rtimes_{\sigma}\ntorus)\rtimes_{\alpha_{1}}\mb{Z}\rtimes\cdots\rtimes_{\alpha_{n}}\mb{Z} \lra A\rtimes_{\sigma}\ntorus\rtimes_{\gamma_{1}^{t}}\mb{Z}\rtimes\cdots\rtimes_{\gamma_{n}^{t}}\mb{Z}.
\end{equation}

For each $t\in [0,1]$ let $I_{t} = \{f\in C([0,1]) \, \lvert \quad f(t) = 0\}$ be the ideal of functions vanishing at the point $t$. The ideal $I_{t}\otimes A\rtimes_{\sigma}\ntorus \subseteq C([0,1])\otimes A\rtimes_{\sigma}\ntorus$ is $\alpha_{1}$-invariant, so it follows that the kernel of the *-homomorphism $\pi_{t}$ in $(\ref{sec5eqBundleStep1})$ is
\[ ker\,\pi_{t} = (I_{t}\otimes A\rtimes_{\sigma}\ntorus)\rtimes_{\alpha_{1}}\mb{Z}. \]
By iteration, it follows that the kernel of the *-homomorphism $\pi_{t}$ in $(\ref{sec5eqBundle})$ is
\[ ker\,\pi_{t} = (I_{t}\otimes A\rtimes_{\sigma}\ntorus)\rtimes_{\alpha_{1}}\mb{Z}\rtimes\cdots\rtimes_{\alpha_{n}}\mb{Z}. \]

Using a homeomorphism of $[0,1]$ to itself, mapping $t$ to $1$, there is a *-isomorphism $I_{t} \cong C_{0}([0,1))$. This means $I_{t}\otimes A\rtimes_{\sigma}\ntorus \cong C_{0}([0,1))\otimes A\rtimes_{\sigma}\ntorus = Cone(A\rtimes_{\sigma}\ntorus)$, hence
\begin{equation}\label{kerCone}
ker\,\pi_{t} = Cone(A\rtimes_{\sigma}\ntorus)\rtimes_{\alpha_{1}}\mb{Z}\rtimes\cdots\rtimes_{\alpha_{n}}\mb{Z}.
\end{equation}

We recall a few general facts which we will appeal to shortly, in particular contractibility of cones and the Pimsner-Voiculescu six-term exact sequence. First, a C*-algebra $B$ is called \textit{KK-contractible} if $KK(B,B) = 0$. This also implies $KK(B,D) = 0 = KK(D,B)$ for any other C*-algebra $D$. 

Suppose there is an action $\beta\in Aut(B)$. The Pimsner-Voiculescu six-term exact sequence in KK-theory is
\[
\begin{CD}
KK(D,B) @>1-\beta_{*}>> KK(D,B) @>>> KK(D,B\rtimes_{\beta}\mb{Z}) \\
@AAA @. @VVV \\
KK^{1}(D, B\rtimes_{\beta}\mb{Z})  @<<< KK^{1}(D,B) @<<1-\beta_{*}< KK^{1}(D,B)
\end{CD}
\]
Observe that if $B$ is KK-contractible, then the six-term exact sequence reads
\[
\begin{CD}
0 @>1-\beta_{*}>> 0 @>>> KK(D,B\rtimes_{\beta}\mb{Z}) \\
@AAA @. @VVV \\
KK^{1}(D, B\rtimes_{\beta}\mb{Z})  @<<< 0 @<<1-\beta_{*}< 0
\end{CD}
\]
and using in particular $D = B\rtimes_{\beta}\mb{Z}$ we deduce $KK(B\rtimes_{\beta}\mb{Z}, B\rtimes_{\beta}\mb{Z}) = 0$, i.e. $B\rtimes_{\beta}\mb{Z}$ is KK-contractible.

Given any separable C*-algebra $D$, its cone $Cone(D) = C_{0}([0,1))\otimes D$ is KK-contractible.

\begin{theorem} For every $t\in [0,1]$ the bundle map
\[ \pi_{t}: C([0,1], A\rtimes_{\sigma}\ntorus)\rtimes_{\alpha_{1}}\mb{Z}\rtimes\cdots\rtimes_{\alpha_{n}}\mb{Z} \lra A\rtimes_{\sigma}\ntorus\rtimes_{\gamma_{1}^{t}}\mb{Z}\rtimes\cdots\rtimes_{\gamma_{n}^{t}}\mb{Z} \]
gives a KK-equivalence.
\end{theorem}
\begin{proof} From $(\ref{kerCone})$ $ker\,\pi_{t} = Cone(A\rtimes_{\sigma}\ntorus)\rtimes_{\alpha_{1}}\mb{Z}\rtimes\cdots\rtimes_{\alpha_{n}}\mb{Z}$. Then the KK-contractibility of $Cone(A\rtimes_{\sigma}\ntorus)$ combined with a repeated Pimsner-Voiculescu six-term sequence argument as above establishes that $ker\,\pi_{t}$ is KK-contractible. This implies that $\pi_{t}$ gives a KK-equivalence element.
\end{proof}

\subsection{Invariance of the index}
The index pairing is the pairing between K-theory and K-homology
\[ K_{0}(A)\times K^{0}(A) \lra \mb{Z} \]
\begin{equation}\label{indexParingEq}
\langle [e], [(\mc{H},F)]\rangle = index\left(e(F^{+}\otimes 1_{k})e: e\mc{H}^{k} \lra e\mc{H}^{k}\right),
\end{equation}
for a projection $e\in M_{k}(A)$ and Fredholm module $(\mc{H},F)$ for $A$. This pairing is nothing but the KK-product
\begin{equation}\label{indexPairingKKprod}
KK(\mb{C},A) \times KK(A,\mb{C}) \lra KK(\mb{C},\mb{C}) 
\end{equation}
after the identifications $K_{0}(A) = KK(\mb{C},A)$, $K^{0}(A) = KK(A,\mb{C})$ and $KK(\mb{C},\mb{C}) = \mb{Z}$.

See also \cite{Yamashita} for a discussion of theta deformation and the invariance of the index, and moreover a calculation of the Chern character map for the deformation.

Given an even spectral triple $(A,\mc{H},D)$ there is the associated Fredholm module $(\mc{H}, F)$ with $F = D|D|^{-1}$.

Our separable C*-algebra $A$ is assumed equipped with an action $\sigma:  \ntorus\lra Aut(A)$, and let $\mc{A}\subseteq A$ be the dense *-subalgebra of smooth elements for the action. Suppose $(\mc{A},\mc{H},D)$ is a spectral triple, with a *-representation $\varphi: A\lra B(\mc{H})$. Assume the action to be unitarily implemented by $U:\ntorus \lra B(\mc{H})$, $\varphi(\sigma_{s}(a)) = U_{s}\varphi(a)U_{s}^{*}$, and that $U_{s}D = DU_{s}$ for each $s\in\ntorus$.

Theta deformation is an \textit{isospectral} deformation, meaning that the same data $(\mc{H}, D)$ which describes a noncommutative geometry for $A$, is also taken to serve a noncommutative geometry for $A_{\theta}$. In order to study these aspects, it is useful to work with the following picture of the deformation. Any element $a\in \mc{A}$ decomposes into a norm convergent series $a = \sum_{r\in\Zn} a_{r}$ where each $a_{r}\in \mc{A}$ satisfies $\sigma_{s}(a_{r}) = e^{-2\pi i r\cdot s}a_{r}$, for $s\in\ntorus$. Given two elements $a, b\in \mc{A}$ with decompositions $a = \sum_{r}a_{r}$ and $b=\sum_{p}b_{p}$, the product $\times_{\theta}$ takes the form
\begin{equation}\label{thetaProd}
a_{r}\times_{\theta}b_{p} = e^{2\pi i r\cdot\theta p}a_{r}b_{p} 
\end{equation}
between two component elements $a_{r}$ and $b_{p}$. The product $a\times_{\theta}b$ is then the linear extension of the componentwise product $(\ref{thetaProd})$. The *-algebra $\mc{A}_{\theta}$ is just $\mc{A}$ equipped with this product. The correspondence with the definition in $(\ref{thetaDeformation})$ is just
\[ a_{r} \mapsto a_{r} \otimes u_{1}^{r_{1}}\cdots u_{n}^{r_{n}} \in (A\otimes C(\ntorus_{\theta}))^{\sigma\otimes\tau^{-1}}. \]
We have a representation $\varphi_{\theta}$ of $\mc{A}_{\theta}$ on the same Hilbert space, $\varphi_{\theta}: \mc{A}_{\theta} \lra B(\mc{H})$, by
\[ \varphi_{\theta}(a) = \sum_{r} \varphi(a_{r})U_{q(\theta r)}, \]
$A_{\theta}$ is then the norm closure and $(\mc{A}_{\theta}, \mc{H}, D_{\theta})$ is the deformed spectral triple, where $D_{\theta} = D$.

For the even spectral triple $(A,\mc{H},D)$ we shall denote by $[D] = [(\mc{H}, \varphi, F)]\in K^{0}(A)$ the corresponding element of K-homology. Likewise we denote by $[D_{\theta}] = [(\mc{H}, \varphi_{\theta}, F)]\in K^{0}(A_{\theta})$ the element associated to the spectral triple $(A_{\theta},\mc{H},D_{\theta})$.

\begin{corollary} The KK-equivalence of Theorem $\ref{maintheorem}$ induces an isomorphism $K^{0}(A) \cong K^{0}(A_{\theta})$ mapping
 $[D] \mapsto [D_{\theta}]$.
\end{corollary}
\begin{proof} Let $\Gamma = \Gamma( (A_{t\theta})_{t\in[0,1]} )$. From the bundle maps $\pi_{0}: \Gamma\lra A$ and $\pi_{1}: \Gamma\lra A_{\theta}$ we get by Theorem $\ref{maintheorem}$ the KK-equivalence elements $[\pi_{0}]\in KK(\Gamma, A)$ and $[\pi_{1}]\in KK(\Gamma, A_{\theta})$. The relevant mappings between KK-groups is described by the KK-products
\[
\begindc{\commdiag}
\obj(3,3)[L1]{$KK(\Gamma,\mb{C})$}
\obj(1,1)[L2]{$KK(A,\mb{C})$}
\obj(5,1)[L3]{$KK(A_{\theta},\mb{C})$}
\mor{L2}{L1}{\scriptsize$[\pi_{0}]\cdot$}
\mor{L3}{L1}{\scriptsize$[\pi_{1}]\cdot$}[-1,0]
\mor{L2}{L3}{\scriptsize$[\pi_{1}]^{-1}\cdot [\pi_{0}]\cdot$}[-1,1]
\enddc
\]
where $[\pi_{0}] = [(A, \pi_{0}, 0)] \in KK(A,\mb{C})$ and $[\pi_{1}] = [(A_{\theta}, \pi_{1}, 0)]\in KK(\Gamma,A_{\theta})$ are the KK-cycle descriptions.

The element $[D] = [(\mc{H},\varphi, F)]\in KK(A,\mb{C})$ is the element canonically associated to the given spectral triple $(A,\mc{H},D)$ as explained above, and upon taking the KK-product we get
\begin{equation}\label{indexEq1}
[\pi_{0}]\cdot [D] = [(\mc{H},\varphi\circ\pi_{0}, F)]\in KK(\Gamma, \mb{C}).
\end{equation}
Likewise $[D_{\theta}] = [(\mc{H}, \varphi_{\theta}, F)]\in KK(A_{\theta},\mb{C})$ is the element associated to the deformed spectral triple $(A_{\theta},\mc{H},D_{\theta})$, and the KK-product is then
\begin{equation}\label{indexEq2}
[\pi_{1}]\cdot [D_{\theta}] = [(\mc{H},\varphi\circ\pi_{1}, F)]\in KK(\Gamma, \mb{C}).
\end{equation}
It will be enough to establish the equality $[\pi_{0}]\cdot [D] = [\pi_{1}]\cdot [D_{\theta}]$ in $KK(\Gamma,\mb{C})$. This follows from homotopy of KK-cycles.
Indeed, let $(E,\phi, F)\in KK(\Gamma, I\mb{C})$ be the element where $E = C([0,1],\mc{H})$, $\phi: \Gamma\lra \mc{L}_{I\mb{C}}(E)$, $(\phi(s)\xi)(t) = s(t)\xi(t)$, and $I\mb{C} = C([0,1])\otimes\mb{C} = C([0,1])$. Let $ev_{0}$ and $ev_{1}$ denote the respective evaluation morphisms $E\lra \mc{H}$. It is easy to check (using details explained in \cite{Connes-Landi}) that $(E,\phi,F)$ provides a homotopy between the KK-cycles $(\ref{indexEq1})$ and $(\ref{indexEq2})$, i.e. isomorphisms of the KK-cycles with the pushouts of $ev_{0}$ and $ev_{1}$ respectively,
\[ (E_{ev_{0}}, \phi_{ev_{0}}, F_{ev_{0}}) \cong [(\mc{H},\varphi\circ\pi_{0}, F)] \quad\mbox{ and }\quad (E_{ev_{1}}, \phi_{ev_{1}}, F_{ev_{1}}) \cong [(\mc{H},\varphi\circ\pi_{1}, F)]. \]
\end{proof}

The KK-equivalence of Theorem $\ref{maintheorem}$ implies the isomorphisms
\[ K_{0}(A) = KK(\mb{C}, A) \lra KK(\mb{C},A_{\theta}) = K_{0}(A_{\theta}), \quad [e] \longmapsto [e]\cdot [\pi_{0}]^{-1}\cdot [\pi_{1}], \]
and
\[ K^{0}(A) = KK(A, \mb{C}) \lra KK(A_{\theta},\mb{C}) = K^{0}(A_{\theta}), \quad [(\mc{H},F)] \longmapsto [\pi_{1}]^{-1}\cdot [\pi_{0}]\cdot [(\mc{H},F)], \]
and regarding the index pairing $(\ref{indexParingEq})$ or equivalently the KK-product $(\ref{indexPairingKKprod})$, we get  
\[
\begin{CD}
K_{0}(A)\times K^{0}(A) @>\mbox{index}>> \mb{Z} \\
@VVV @VV{\lvert\lvert}V \\
K_{0}(A_{\theta})\times K^{0}(A_{\theta}) @>>\mbox{index}> \mb{Z}
\end{CD}
\]
where $[e]\cdot [(\mc{H},F)]$ is the top index pairing and 
\[ [e]\cdot [\pi_{0}]^{-1}\cdot [\pi_{1}] \cdot [\pi_{1}]^{-1}\cdot [\pi_{0}]\cdot [(\mc{H},F)] = [e]\cdot [(\mc{H},F)] \]
is the bottom index pairing after having followed the isomorphisms induced by the KK-equivalences.

\author{Amandip Sangha\\
Department of Mathematics,\\
  University of Oslo,\\
  PO Box 1053 Blindern,\\
  N-0316 Oslo, Norway.\\
  \texttt{amandip.s.sangha@gmail.com}}



\newpage
\begin{thebibliography}{20}
\bibitem{Connes-DuboisViolette} A. Connes, M. Dubois-Violette. \emph{Noncommutative finite-dimensional manifolds I. Spherical manifolds and related examples}, arXiv:math/0107070v5 [math.QA], 2002.
\bibitem{Connes-Landi} A. Connes, G. Landi. \emph{Noncommutative Manifolds the Instanton Algebra and Isospectral Deformations}, arXiv:math/0011194v3 [math.QA], 2001.
\bibitem{Echterhoff-Nest-Oyono-1} S. Echterhoff, R. Nest, H. Oyono-Oyono. \emph{Principal noncommutative torus bundles}, arXiv:0810.0111v1  [math.KT], 2008.
\bibitem{Echterhoff-Nest-Oyono-2} S. Echterhoff, R. Nest, H. Oyono-Oyono. \emph{Fibrations with noncommutative fibers}, arXiv:0810.0118v1 [math.KT], 2008.
\bibitem{H-M} K. C. Hannabuss, V. Mathai, \emph{Noncommutative principal torus bundles via parametrised strict deformation quantization}, arXiv:0911.1886v2 [math-ph] , 2010.
\bibitem{Kasparov1} G. G. Kasparov. \emph{Equivariant KK-theory and the Novikov conjecture}, Invent. math 91, p147-201 (1988)
\bibitem{kasprzak} P. Kasprzak, \emph{Rieffel deformation via crossed products}, 	arXiv:math/0606333v15 [math.OA], 2010.
\bibitem{K-W} E. Kirchberg, S. Wassermann. \emph{Operations on continuous bundles of C*-algebras}, Math. Ann. 303(1995), p.677-697
\bibitem{Landi} G. Landi. \emph{Examples of noncommutative instantons}, arXiv:math/0603426v2 [math.QA], 2007.
\bibitem{Landstad} M. B. Landstad. \emph{Duality theory for covariant systems}, Trans. Amer. Math. Soc. 248 (1979), no. 2, 223-267
\bibitem{nagy} G. Nagy, \emph{Deformation quantization and K-theory}, Contemporary Mathematics Volume 214 (1998), p.111-134
\bibitem{Ng} C. K. Ng. \emph{Morita equivalences between fixed point algebras and crossed products}, Math. Proc. Cambridge Philos. Soc. 125 (1999), p.43-52
\bibitem{pedersen} G. K. Pedersen. \emph{C*-algebras and their automorphism groups}, Academic Press, 1979.
\bibitem{Rieffel2} M. A. Rieffel. \emph{Continuous fields of C*-algebras coming from group cocycles and actions}, Math. Ann. 283, 631-643 (1989)
\bibitem{Rieffel1} M. A. Rieffel. \emph{Deformation quantization for actions of $\mb{R}^{d}$}, Mem. Am. Math. Soc. 506 (1993)
\bibitem{Rieffel3} M. A. Rieffel. \emph{K-groups of C*-algebras deformed by actions of $\mb{R}^{d}$}, Journal of functional analysis 116 (1993), p.199-214
\bibitem{Yamashita} M. Yamashita. \emph{Connes-Landi Deformation of Spectral Triples}, arXiv:1006.4420v1 [math.OA], (2010)
\end{thebibliography}
\end{document}